\documentclass[11pt,leqno]{amsart}
\usepackage{epsfig}
\usepackage{amssymb}
\usepackage{amscd}
\usepackage[matrix,arrow]{xy}
\usepackage{graphicx}

\newtheorem{theo}{Theorem}[section]
\newtheorem{lemma}[theo]{Lemma}
\newtheorem{defi}[theo]{Definition}
\newtheorem{prop}[theo]{Proposition}
\newtheorem{conj}[theo]{Conjecture}
\newtheorem{cor}[theo]{Corollary}

\numberwithin{equation}{section}

\def\R{\mathbb{R}}
\def\C{\mathbb{C}}
\def\Z{\mathbb{Z}}
\def\Q{\mathbb{Q}}

\def\bR{{\mathbf R}}

\def\pre-tr{\operatorname{pre-tr}}

\def\Hom{\operatorname{Hom}}

\newcommand{\Int}{\operatorname{Int}}

\newcommand{\Area}{\operatorname{Area}}
\newcommand{\Vol}{\operatorname{Vol}}

\newcommand{\cO}{{\mathcal O}}

\newcommand{\supp}{\operatorname{Supp}}
\newcommand{\Ker}{\operatorname{Ker}}
\newcommand{\Pic}{\operatorname{Pic}}
\newcommand{\coker}{\operatorname{Coker}}

\newcommand{\rank}{\operatorname{rank}}

\newcommand{\rk}{\operatorname{rk}}
\newcommand{\conv}{\operatorname{conv}}

\newcommand{\Sign}{\rm Sign}

\newcommand{\Gale}{\operatorname{Gale}}

\newcommand{\id}{\operatorname{id}}

\usepackage{epsf}
\usepackage{amscd}

\title[Maximal lengths of exceptional collections of line bundles]
{Maximal lengths of exceptional collections of line bundles}

\author{Alexander I. Efimov}
\address{Steklov Mathematical Institute of RAS, Gubkin str. 8, GSP-1, Moscow 119991, Russia} \email{efimov@mccme.ru}
\thanks{The author was partially supported by
the Moebius Contest Foundation for Young Scientists, and RFBR (grant 4713.2010.1).}

\begin{document}

\begin{abstract} In this paper we construct infinitely many examples of toric Fano varieties with Picard number three, which do not admit full exceptional collections of line bundles. In particular, this disproves King's conjecture for toric Fano varieties.

More generally, we prove that for any constant $c>\frac34$ there exist infinitely many toric Fano varieties $Y$ with Picard number three,
such that the maximal length of exceptional collection of line bundles on $Y$ is strictly less than $c\rk K_0(Y).$ To obtain varieties without exceptional collections of line bundles,
it suffices to put $c=1.$

On the other hand, we prove that for any toric nef-Fano DM stack $Y$ with Picard number three, there exists a strong exceptional collection of line bundles on $Y$
of length at least $\frac34 \rk K_0(Y).$ The constant $\frac34$ is thus maximal with this property.
\end{abstract}

\maketitle


\section{Introduction}

A conjecture of King \cite{Ki} claims that each smooth projective toric variety has a full strong exceptional collection of line bundles.
It was disproved \cite{HP1, HP2, Mi} in infinitely many cases. However, all the counter-examples were not nef-Fano. Borisov and Hua
have proposed the following modification (and a generalization).

\begin{conj}\label{Borisov_Hua}Every smooth nef-Fano toric DM stack possesses a
full strong exceptional collection of line bundles.\end{conj}

They proved Conjecture \ref{Borisov_Hua} \cite{BH} in the case of Fano stacks for which either Picard number or dimension is at most two.
The case of nef-Fano Del Pezzo stacks was further treated in \cite{IU}.

The weaker version of conjecture was proposed by Costa and Mir\'o-Roig:

\begin{conj}\label{Miro-Roig}For any smooth, complete Fano toric variety there exists
a full, strongly exceptional collection of line bundles.\end{conj}

In this paper we disprove Conjecture \ref{Miro-Roig} (and hence Conjecture \ref{Borisov_Hua}) by proving the following theorem.

\begin{theo}\label{<cintro}For any constant $c>\frac34$ there exist infinitely many toric Fano varieties $Y$ with Picard number three,
such that the maximal length of exceptional collection of line bundles on $Y$ is strictly less than $c\rk K_0(Y).$
In particular (for $c=1$), there are infinitely many toric Fano varieties with Picard number three without full exceptional collections of line bundles.\end{theo}

More precise statement is Theorem \ref{<c}. We also give infinitely many explicit examples of toric Fano varieties
without full exceptional collections of line bundles (Theorem \ref{counterex}). Note that we do not require the collections to be strong.

On the other hand we prove another (positive) result which constructs not full but relatively long, strong exceptional collection of line bundles on each toric nef-Fano DM stack with Picard number three (see Theorem \ref{at least}).

\begin{theo}\label{at_leastintro}For any toric nef-Fano DM stack $Y$ with Picard number three, there exists a strong exceptional collection of line bundles on $Y$
of length at least $\frac34 \rk K_0(Y).$\end{theo}

Recall the result of Kawamata:

\begin{theo}For any smooth projective toric DM stack $Y,$ the derived category $D^b(Y)$ is generated by exceptional collection of coherent sheaves.\end{theo}

The following conjecture was suggested to me by D.~Orlov.

\begin{conj}For any smooth projective toric DM stack $Y,$ the derived category $D^b(Y)$ is generated by a strong exceptional collection.\end{conj}

It remains open.

\smallskip

The paper is organized as follows.

In Section \ref{Gale} we recall some necessary notions and facts about Gale duality.

In Section \ref{DM_stacks} we recall stacky fans and corresponding toric DM stacks.
Here we also describe the conditions on fan corresponding to the (nef-)Fano condition on stack.

Section \ref{cohomology} is devoted to the well-known description of cohomology of line bundles on smooth toric DM stacks (Proposition \ref{Czech}).
Here we also describe the line bundles with zero cohomology and zero higher cohomology (Corollary \ref{non-trivial_cohom}).

Section \ref{var_pic_three} is devoted to the construction of toric Fano varieties with Picard number three in terms of a Gale dual picture.
It is used in the proof of main theorem.

In Section \ref{construction} we construct a certain family of toric Fano varieties $Y_{n,k,a}$ parameterized by integers $n,k\geq 2,$ $a\geq 1.$
We use varieties of this family to prove Theorem \ref{<cintro} (more precisely, Theorem \ref{<c}). The proof is rather technical and
contains a lots of technical bounds. Further, using the proof of Theorem \ref{<c},
we prove that varieties $Y_{16,k,1}$ do not have full exceptional collections of line bundles for $k\geq 386$ (Theorem \ref{counterex}).

In Section \ref{at_least_sect} we prove Theorem \ref{at_leastintro}. The idea is the following. We construct a centrally symmetric polytope $P\subset \Pic_{\R}(Y)$
with the following property: the integral points (treated as line bundles) in the interior of any shift of $\frac12 P$ form a strong exceptional collection.
Further, it follows from Fubini's theorem that for some shift the length of this collection is at least $\frac18\Vol(P).$ Then it remains to prove
the inequality $\frac18\Vol(P)\geq \frac34\rk K_0(Y).$

In Appendix we give a combinatorial description of fans defining smooth toric DM stacks with Picard number three.

{\bf Acknowledgements.} I am grateful to L.~Borisov who pointed out a mistake in the proof of Theorem \ref{<c} in the preliminary version of the paper.
Essentially it does not affect the proof, but the smallest dimension of my counter-examples becomes higher.
I am also grateful to A.~Craw for his remarks and to A.~Kuznetsov and D.~Orlov for useful discussions.

\section{Gale duality}
\label{Gale}

In this section we remind some basic notions and facts related to Gale duality.

Let $V$ be a finite-dimensional vector space over $\R,$ and $v_1,\dots,v_n$ a finite collection of vectors which generate $V.$
Then we have the surjection

\begin{equation}\label{surj_1}p:\R^N\twoheadrightarrow V, e_i\mapsto v_i, i=1,\dots,N,\end{equation}
where $e_1,\dots,e_n$ are standard basis vectors.
Take the dual injection $p^{\vee}:V^{\vee}\hookrightarrow \left(\R^N\right)^{\vee}\cong\R^N,$ and the corresponding quotient map
\begin{equation}\label{surj_2}q:\R^N\to\R^N/V^{\vee}=:U.\end{equation}
Also put $E_i:=q\left(e_i\right)\in U,$ $1\leq i\leq N.$

\begin{defi}In the above notation, the surjection \eqref{surj_2} is called Gale dual to the surjection \eqref{surj_1}.
Further, the collection of vectors $E_1,\dots,E_n\in U$ is called Gale dual to the collection
$v_1,\dots,v_n\in V.$\end{defi}

It is clear that the surjection \eqref{surj_1} is canonically identified with Gale dual to \eqref{surj_2}. Further, it follows from the definition that
we have the following isomorphisms:
\begin{equation}\label{funct_V_rel_E}\{\text{linear functionals on }V\}\cong\{\text{linear relations on }E_1,\dots,E_n\},\end{equation}
\begin{equation}\label{funct_U_rel_v}\{\text{linear functionals on }U\}\cong\{\text{linear relations on }v_1,\dots,v_n\}.\end{equation}
For instance, linear functional $l\in V^{\vee}$ gives a linear relation $l\left(v_1\right)E_1+\dots+l\left(v_n\right)E_n=0.$

We would like to reformulate some statements about vectors $v_i$ in terms of vectors $E_i.$

\begin{prop}\label{contains_0}The following are equivalent:

(i)  the interior of the convex hull of $v_1,\dots,v_n$ contains the origin;

(ii) there exists a functional $l\in U^{\vee}$ such that $l\left(E_i\right)>0$ for $i=1,\dots,n.$\end{prop}

\begin{proof}First, (i) is equivalent to the existence of linear relation
$$a_1v_1+\dots+a_nv_n=0,\quad a_1,\dots,a_n>0.$$
And this is in turn equivalent to (ii) by \eqref{funct_U_rel_v}.\end{proof}

Consider the vector $K=-E_1-\dots-E_n\in U.$ Further, put $\overline{U}:=U/(\R\cdot K),$ and let $\overline{E_i}\in \overline{U}$
be the projection of $E_i\in U$ for $i=1,\dots,n.$

\begin{prop}\label{is_convex}Suppose that both equivalent statements in Proposition \ref{contains_0} hold. Then the following are equivalent:

(i) points $v_i\in V$ are vertices of a convex polytope;

(ii) for each $i=1,\dots, n$ there exist positive numbers $a_1,\dots,\widehat{a_i},\dots,a_n$ such that
$\sum\limits_{j\ne i}a_jE_j=-K;$

(iii) for each $i=1,\dots, n,$ the interior of the convex hull of $\overline{E_1},\dots,\widehat{\overline{E_i}},\dots,\overline{E_n}$ contains the origin.
\end{prop}

\begin{proof}$(i)\Rightarrow (ii)$ Since this convex polytope contains zero (by assumption), for each $i=1,\dots,N$ there exists a functional $l_i\in V^{\vee},$ such that $l_i\left(v_i\right)=1$ and $l_i\left(v_j\right)<1$ for $j\ne i.$ By \eqref{funct_V_rel_E}, this functional gives a relation on $E_j$ which can be rewritten as follows:
$$\sum\limits_{j\ne i}\left(1-l_i(v_j)\right)E_j=-K.$$ The implication is proved.

$(ii)\Rightarrow (i)$ This implication is proved analogously to the previous one.

$(ii)\Rightarrow (iii)$ We have that $\sum\limits_{j\ne i}a_j\overline{E_j}=0$ with $a_j>0,$ hence the assertion.

$(iii)\Rightarrow (ii)$ For each $i=1,\dots,n$ there exists positive numbers $b_j,$ $j\ne i,$ such that $\sum\limits_{j\ne i}b_j\overline{E_j}=0.$
Then $\sum\limits_{j\ne i}b_jE_j=aK$ for some $a\in\bR.$
Take some functional $l\in U^{\vee}$ such that $l\left(E_j\right)>0$ for $1\leq j\leq n$ (such $l$ exists by our assumption). Then
$$l\left(aK\right)=\sum\limits_{j\ne i}b_jl\left(E_j\right)>0.$$
Since $l\left(K\right)<0,$ we have that $a<0.$ Put $a_j:=-\frac{b_j}a,$ $j\ne i.$ Then we have $\sum\limits_{j\ne i}a_jE_j=-K$ and $a_j>0.$
\end{proof}

\begin{prop}\label{compl_to_faces}Suppose that all statements of Propositions \ref{contains_0} and \ref{is_convex} hold. Let $B\subset \{1,\dots,n\}$ be a non-empty subset, and $\overline{B}\subset \{1,\dots,n\}$
its complement. Then the following are equivalent:

(i) the set $\{v_j,j\in B\}$ is the set of vertices of some face of the polytope;

(ii) there exist positive numbers $a_j,$ $j\in\overline{B}$ such that
$\sum\limits_{j\in\overline{B}}a_jE_j=-K;$

(iii) the relative interior of the convex hull of $\bar{E_j},$ $j\in\overline{B},$ contains the origin.\end{prop}

\begin{proof}The proof goes absolutely analogously to Proposition \ref{is_convex}.\end{proof}

Recall that a convex polytope is called simplicial if all its facets (and hence all faces) are simplices. We have the following corollary.

\begin{cor}Suppose that all statements of Propositions \ref{contains_0} and \ref{is_convex} hold. Then the following are equivalent:

(i) the convex hull of $v_1,\dots,v_n$ is a simplicial polytope;

(ii) for any subset $A\subset \{1,\dots,N\}$ with $|A|<\dim U=N-\dim V,$ the convex hull of $\overline{E_j},$ $j\in A,$
does not contain the origin.\end{cor}

Now note that we have natural isomorphism $\det(V)\cong \det(\R^n/V^{\vee})=\det(U)$ of one-dimensional spaces. Fix some
volume forms $\omega$ on $V$ and $\omega'$ on $U$ which correspond to each other.

\begin{lemma}\label{volumes}Choose a permutation $\sigma\in S_n.$ Then we have
$$|\omega(v_{\sigma(1)},\dots,v_{\sigma(\dim V)})|=|\omega'(E_{\sigma(\dim V+1)},\dots,E_{\sigma(n)})|.$$\end{lemma}

\begin{proof}We may and will assume that $\sigma=\id.$ Take the dual vollume form $\omega^{\vee}$ on $V^{\vee}$ and choose functionals
$f_1,\dots,f_{\dim V}\in V^{\vee}$ such that $\omega^{\vee}(f_1,\dots,f_{\dim V})=1.$ Then we have the following chain of equalities:
\begin{multline*}|\omega(v_1,\dots,v_{\dim V})|=|\omega(v_1,\dots,v_{\dim V})|\cdot |\omega^{\vee}(f_1,\dots,f_{\dim V})|=
|\det(f_i(v_j))_{1\leq i,j\leq\dim V}|\\
=|\det(p^{\vee}(f_1),\dots,p^{\vee}(f_{\dim V}),e_{\dim V+1},\dots,e_n)|\\
=|\omega^{\vee}(f_1,\dots,f_{\dim V})|\cdot
|\omega'(E_{\dim V+1},\dots,E_n)|=|\omega'(E_{\dim V+1},\dots,E_n)|.\end{multline*}
Lemma is proved.\end{proof}

Note that Gale duality can be also considered for integer lattices so that after tensoring with $\R$ we obtain the above picture.
More precisely, we do not assume that $v_1,\dots,v_n\in C$ generate the lattice $C$ but we still assume that they generate the real space $C_{\R}.$
The Gale dual collection $E_1,\dots,E_n\in D$ generates the abelian group $D.$ It may have torsion:
\begin{equation}\label{D_tors}D_{tors}\cong \Hom(C/(\Z\cdot v_1+\dots+\Z\cdot v_n),\C^*).\end{equation}

\begin{lemma}\label{vol_D_tors}Assume that $V=C_{\R},$ where $C$ is some integral lattice, and we have $v_i\in C.$ Choose vollume
form $\omega$ in such a way that the volume of unit parallelepiped of $C$ equals to $1.$ Let $D$ be a Gale dual lattice.
Then the volume of unit parallelepiped of $D/D_{tors}$ (with respect to $\omega'$ on $D_{\R}=U$) equals to
$$|D_{tors}|.$$\end{lemma}

\begin{proof}We may and will assume that $D_{tors}=0$ (by replacing $C$ with $\Z\cdot v_1+\dots+\Z\cdot v_n$ and $\omega$
with $|D_{tors}|\cdot\omega,$ according to
\eqref{D_tors}). Choose any basis $f_1,\dots,f_{\dim V}$ of $\Hom(C,\Z),$ and choose any $u_1,\dots,u_{\dim U}\in \Z^N$
such that $q(u_1),\dots,q(u_{\dim U})$ form a basis of $D.$ Then
$$|\omega'(q(u_1),\dots,q(u_{\dim U}))|=|\det(u_1,\dots,u_{\dim U},p^{\vee}(f_1),\dots,p^{\vee}(f_{\dim V}))|=1$$
since $p^{\vee}(f_1),\dots,p^{\vee}(f_{\dim V}),u_1,\dots,u_{\dim U}$
generate $\Z^N.$\end{proof}

We have the following

\begin{cor}\label{is_basis}Let $v_1,\dots,v_n\in C$ be a collection of vectors in integer lattice which generate it. Let $E_1,\dots,E_n\in D$
be a Gale dual collection. Further, let $A\subset \{1,\dots,N\}$ be a subset with $|A|=\rank\left(C\right),$ and $\overline{A}\subset \{1,\dots,N\}$
its complement. Then the following are equivalent:

(i) the vectors $v_j,$ $j\in A,$ generate the lattice $C;$

(ii) the vectors $E_j,$ $j\in \overline{A},$ generate the lattice $D.$\end{cor}

\begin{proof}This follows immediately from Lemmas \ref{volumes} and \ref{vol_D_tors}.
\end{proof}

\section{Smooth toric DM stacks}
\label{DM_stacks}

Let $N$ be a free finitely generated abelian group, and let $\Sigma$ be a complete simplicial fan in $N.$
We call $\Sigma$ a stacky fan if on any one-dimensional cone $\sigma\in\Sigma(1)$ there chosen a non-zero vector $v_{\sigma}\in\sigma\cap N.$

The associated toric DM stack $Y=Y_{\Sigma}$ is constructed as follows. We have natural
surjection
$$\Z^{\Sigma(1)}\to N,\quad e_{\sigma}\mapsto v_{\sigma}.$$
Put
$$\Gale(N):=\coker(N^{\vee}\to \Z^{\Sigma(1)}).$$
Define the algebraic Group $G$ by the formula
$$G:=\Hom(\Gale(N),\C^*).$$
Define the open subset $U\subset \C^{\Sigma(1)}$ as follows. The point $z\in\C^{\Sigma(1)}$ lies in $U$ if the set $\{\sigma\in\Sigma(1)\mid z_{\sigma}=0\}$
is not a set of one-dimensional cones of some cone in $\Sigma.$

We have a natural action of $G$ on $U$ via inclusion $G\subset (\C^*)^{\Sigma(1)}.$ Put
$$Y_{\Sigma}:=U/G.$$\

The stack $Y_{\Sigma}$ is smooth and complete. The torus
$$T=(\C^*)^{\Sigma(1)}/G$$
naturally acts on $Y_{\Sigma}.$ The orbits of codimension $i$ are in bijection with cones $\sigma\in\Sigma(i):$
$$\sigma\leftrightarrow\{z\in U\mid z_l=0\text{ for }l\subset\partial\sigma,\quad z_l\ne 0\text{ for }l\not\subset\partial\sigma\}/G.$$
If in addition for each maximal cone the vectors $v_{\sigma}$ on its boundary form a basis of $N,$
then $Y_{\Sigma}$ is just a toric variety. The following is well-known (see \cite{FLTZ}, Theorem 4.4):

\begin{prop}\label{Fano-cond}The stack $Y_{\Sigma}$ is nef-Fano (resp. Fano) iff the polytope
$$\bigcup\limits_{\langle v_{i_1},\dots,v_{i_{\rk N}}\rangle\in\Sigma(\rk N)}
\conv(v_{i_1},\dots,
v_{i_{\dim Y}},0)$$
is convex (resp. in addition all $v_{\sigma}$ are its vertices and it is simplicial).
\end{prop}

We will also need the following formula for the rank of $K_0(Y_{\Sigma})$ \cite{BHo}:

\begin{equation}\label{K_0_general}\rk K_0(Y_{\Sigma})=(\rk N)!\Vol(\bigcup\limits_{\langle v_{i_1},\dots,v_{i_{\rk N}}\rangle\in\Sigma(\rk N)}
\conv(v_{i_1},\dots,
v_{i_{\dim Y}},0)).\end{equation}

In particular, if $Y_{\Sigma}$ is a variety, then $\rk K_0(Y_{\Sigma})$ equals to the number of maximal cones in $\Sigma$ (or, equivalently,
torus-invariant points in $Y_{\Sigma}$).

\section{Cohomology of line bundles on smooth toric DM stacks}
\label{cohomology}

Let $\Sigma$ be a complete simplicial stacky fan, and $Y=Y_{\Sigma}$ the corresponding stack.
We have that
$$\Pic(Y)=\Pic_{G}(U)=\Hom(G,\C^*)\cong\Gale(N).$$
Denote by $\{\cO(E_{\sigma})\in\Pic(Y)\}_{\sigma\in\Sigma(1)}$ the Gale dual collection to $\{v_{\sigma}\in N\}_{\sigma\in\Sigma(1)}.$
Then $\cO(E_{\sigma})$ is a line bundle of invariant divisor corresponding to $\sigma.$
In the next sections we will not distinguish divisors and the corresponding line bundles.

Further, for any $I\subset\Sigma(1)$ denote by $C_I$ the simplicial complex with the vertex set $I,$ which consists
of subsets $J\subset I$ which are precisely boundary cones of some cone in $\Sigma.$ For instance, $|C_{\emptyset}|=\emptyset$
and $|C_{\Sigma(1)}|$ is homeomorphic to $S^{\rk N-1}.$

Also, for any $r\in\Z^{\Sigma(1)},$ put
$$\supp(r)=\{\sigma\in\Sigma(1)\mid r_{\sigma}<0\}.$$

The following is well-known (computation by \v{C}ech).

\begin{prop}\label{Czech}Let $L$ be a line bundle on $Y.$ Then
$$H^i(L)=\bigoplus_{\substack{r\in\Z^{\Sigma(1)},\\
\cO(\sum\limits_{\sigma\in\Sigma(1)}r_{\sigma}E_{\sigma})\cong L}}\bar{H}_{i-1}(|C_{\supp(r)}|).$$\end{prop}

For any $I\subset \Sigma(1)$ such that $\bar{H}_{\cdot}(|C_I|)\ne 0,$ we put
\begin{equation}K_I:=\{\cO(\sum\limits_{\sigma\in I}(-r_{\sigma}-1)E_{\sigma}+\sum\limits_{\sigma\notin I}r_{\sigma}E_{\sigma})\mid r_{\sigma}\in\Z_{\geq 0},\sigma\in\Sigma(1)\}\subset\Pic Y.\end{equation}
We call such $K_I$ forbidden sets. For instance, $K_{\emptyset}$ is the set of all effective line bundles.

\begin{cor}\label{non-trivial_cohom}Let $L$ be a line bundle on $Y.$ The following are equivalent:

(i)$H^{\cdot}(L)=0$ (resp. $H^{>0}(L)=0$);

(ii)$L$ does not belong to any forbidden $K_I$ (resp. to any forbidden $K_I,$ $I\ne\emptyset$).\end{cor}

\begin{proof}This follows immediately from Proposition \ref{Czech}.\end{proof}

We will need the notion of a primitive collection.

\begin{defi}A non-empty subset $I\subset \Sigma(1)$ is called a primitive collection if
it is not a set of boundary cones of any cone in $\Sigma,$ but each proper subset $J\subset I$ is.\end{defi}

Note that primitive collections describe the combinatorial structure of a fan (i.e. the corresponding simplicial complex).

\begin{lemma}\label{union}Let $I\subset\Sigma(1)$ be a non-empty subset such that $\bar{H}^{cdot}(|C_I|)\ne 0.$
Then $I$ is a union of primitive collections.\end{lemma}

\begin{proof}Consider the equivariant Picard group $\Pic_T(Y)\cong \Z^{\Sigma(1)},$ with basis given by $\cO(E_i)$ with obvious equivariant structures.
Then a computation by \v{C}ech shows that
$$H^i_T(\cO(\sum\limits_{\sigma\in\Sigma(1)}r_{\sigma}E_{\sigma}))\cong \bar{H}_{i-1}(|c_{\supp(r)}|).$$

Now consider our subset $I.$ Take some element $i\in I.$ We need to prove that there exists a primitive collection $J\subset I$ such that $i\in J.$
If $I$ is a primitive collection itself, then there is nothing to prove. Otherwise,
there exists some proper subset $J\subset I$ which is a primitive collection.

If $i\in J,$ then we are done. Otherwise,
the (twisted by $\cO(-E_i)$) Koszul complex
$$\cO(-E_i)\otimes \bigotimes_{j\in I\setminus\{i\}}(\cO(-E_j)\to\cO)$$
of $T$-equivariant vector bundles is acyclic. Since
$$H^{>0}(\cO(-\sum\limits_{j\in I}E_j))\cong \bar{H}_{\cdot}(|C_I|)\ne 0,$$
we have that for some proper $J\subset I$ containing $i,$
$$H^{>0}(\cO(-\sum\limits_{j\in I}E_j))\cong \bar{H}_{\cdot}(|C_I|)\ne 0.$$
We may replace $I$ by $I'.$ Iterating, we will come to some primitive $I^{\prime\prime}\subset I$ containing $i.$ This proves Lemma.\end{proof}

\section{Toric Fano varieties with Picard number three}
\label{var_pic_three}

Let $\Sigma$ be a stacky fan in some integer lattice $C,$ such that the corresponding stack
$Y:=Y_{\Sigma}$ is a Fano variety with Picard number three. By the description of Batyrev \cite{Ba}, there exists
a decomposition $\Sigma(1)=X_0\sqcup\dots\sqcup X_{2t}$ with $t\in \{1,2\}$ and $X_i\ne\emptyset,$ such
that primitive collections in $\Sigma(1)$ are precisely $X_i\cup\dots\cup X_{i+t-1},$ $0\leq i\leq 2t,$ where
we put $X_{i+2t+1}:=X_i.$ In the case $t=1,$ the King' conjecture was proved (more generally, it was proved for all projective toric varieties
with disjoint primitive collections) \cite{CM}, Theorem 1.3.

We will deal with the case $t=2.$ Note that complements to maximal cones are precisely sets of the form $\{p,q,r\},$ where for some $i$
$p\in X_i,$ $q\in X_{i+1},$ $r\in X_{i+3}.$ Hence, by \eqref{K_0_general}, we have the following formula for the rank of $K_0(Y).$
\begin{equation}\label{rk_K_0}\rk K_0(Y)=\sum\limits_{i=0}^4 |X_i|\cdot|X_{i+1}|\cdot|X_{i+3}|.\end{equation}

Let $E_i\in \Pic Y,$ $i\in \Sigma(1),$ be invariant divisors corresponding to one-dimensional cones. Note that they
determine the vectors $v_i$ via Gale duality. Further, the vectors $v_i$ determine the fan by the Fano condition (Proposition \ref{Fano-cond}).

\begin{prop}\label{descr_by_divisors}Let $E_1,\dots,E_N\in\Z^3$ be a collection of vectors which generate the lattice. Suppose that the following conditions hold:

1) There exists a functional $l\in \left(\R^3\right)^{\vee}$ such that $l\left(E_j\right)>0$ for $j=1,\dots,N;$

2) There exists a decomposition $\{1,\dots,N\}=X_0\sqcup\dots\sqcup X_4$ (the numeration is cyclic),
$X_i\ne\emptyset,$ and functionals $l_0,\dots,l_4\in \left(\R^3\right)^{\vee}$ such that
$$l_i\left(E_j\right)\begin{cases} >0 & \text{for }j\in X_i\cup X_{i+1};\\
<0 & \text{otherwise},\end{cases}$$
and $l_i\left(E_1+\dots+E_N\right)=0.$

3) For each $i=0,\dots,4,$ all triples $p\in X_i,$ $q\in X_{i+1},$ $r\in X_{i+3},$
the vectors $E_p, E_q, E_r$ form a basis of $\Z^3.$

Then the Gale dual collection of vectors defines
a toric Fano variety with five primitive collections $X_i\cup X_{i+1}.$\end{prop}

\begin{proof}Let $v_1,\dots,v_n\in\Z^{N-3}$ be the Gale dual collection.

By condition 1) and Proposition
\ref{contains_0}, the interior of the convex hull of $v_i$ contains the origin.

Put $K:=-E_1-\dots-E_N,$  $W:=\R^3/(\R\cdot K),$ and let $\overline{E_i}\in W$ be the projections of $E_i.$ It follows from
 condition 2) that the following are equivalent:

(i) The interior of the convex hull of $\overline{E_p},\overline{E_q},\overline{E_r}$ contains zero;

(ii) for some $0\leq i\leq 4$ and permutation of $p,q,r,$ we have  $p\in X_i,$ $q\in X_{i+1},$ $r\in X_{i+3}.$

Hence, by Proposition \ref{is_convex}, the points $v_i$ are vertices of a convex polytope. Again by condition 2), for any $1\leq k<l\leq N$ the convex hull of $\overline{E_k},\overline{E_l}$ does not contain the origin. Hence, our convex polytope is simplicial. Further, by equivalence $(i)\Leftrightarrow (ii)$
and Proposition \ref{compl_to_faces}, the complements to (sets of vertices of) facets are of the form $\{p,q,r\},$ $p\in X_i,$ $q\in X_{i+1},$ $r\in X_{i+3}.$
Hence, by condition 3) and Corollary \ref{is_basis}, the vertices of each facet generate the lattice. Therefore,
the vectors $v_i$ define a fan describing toric Fano variety. From the description of maximal cones, we see that primitive collections
are precisely $X_i\cup X_{i+1},$ $i\in \Z/5.$
\end{proof}

\section{Construction of varieties}
\label{construction}

In this section we define a family of toric Fano varieties with Picard number three, parameterized by three positive integers. We will
use it to prove the main theorem.

Take some integers $n\geq 2,k\geq 2,a\geq 1.$ We define five collections of vectors in $\Z^3:$

1) $|X_0|=n+2a,$ $E_{0,1}=\dots=E_{0,n+2a}=(1,0,0);$

2) $|X_1|=1,$ $E_{1,1}=(0,1,0);$

3) $|X_2|=k,$ $E_{2,1}=\dots=E_{2,k-1}=(0,1,1),$ $E_{2,k}=(-a,1,1);$

4) $|X_3|=n,$ $E_{3,1}=\dots=E_{3,n-1}=(0,0,1),$ $E_{3,n}=(-a,0,1);$

5) $|X_4|=1,$ $E_{4,1}=(1,-1,0).$

\begin{prop}For any $n,k,a,$ the Gale dual collection to $X_0\cup\dots\cup X_4$ defines
a toric Fano variety with five primitive collections $X_i\cup X_{i+1}.$\end{prop}

\begin{proof}We will just apply Proposition \ref{descr_by_divisors} and check that all required conditions are satisfied.

First, take the functional $x+\frac{y}2+(a+1)z.$ It is positive on each $E_{i,j},$ hence the condition 1) is satisfied.

Further, we have $-K:=\sum\limits_{i,j}E_{i,j}=(n+1,k,k+n).$ Take the projection
$$\pi:\R^3\to\R^2\cong\R^3/(\R\cdot K),\quad\pi(x,y,z)=\left(\left(k+n\right)x-\left(n+1\right)z,\left(k+n\right)y-kz\right),$$
and put $\overline{E_{i,j}}:=\pi\left(E_{i,j}\right).$ Then we have:

$\overline{E_{0,1}}=\dots=\overline{E_{0,n+2a}}=(k+n,0);$

$\overline{E_{1,1}}=(0,k+n).$

$\overline{E_{2,1}}=\dots=\overline{E_{2,k-1}}=(-n-1,n),$ $\overline{E_{2,k}}=(-ka-na-n-1,n);$

$\overline{E_{3,1}}=\dots=\overline{E_{3,n-1}}=(-n-1,-k),$ $\overline{E_{3,n}}=(-ka-na-n-1,-k)$

$\overline{E_{4,1}}=(k+n,-k-n).$

The required functionals $l_i\in\left(\R^3\right)^{\vee}$ can be defined as pullbacks: $l_i=\pi^*\left(f_i\right),$ where

$f_0=2nx+(2n+1)y;$

$f_1=-x+(ka+na+n+2)y;$

$f_2:=-2x-y;$

$f_3:=-x-(ka+na+n+2)y;$

$f_4:=kx-y.$

Thus, the condition 2) is also satisfied. Finally, condition 3) is checked straightforwardly by computing the determinants.
\end{proof}

Denote by $Y_{n,k,a}$ the toric Fano variety which is obtained from the above Proposition.
For completeness, we describe explicitly the corresponding fan $\Sigma_{n,k,a}.$ To obtain the Gale dual collection to $E_{i,j},$
we choose a basis of additive relations on $E_{i,j}:$
$$E_{0,1}-E_{1,1}-E_{4,1}=0;$$
$$E_{0,j}-E_{0,j+1}=0,\quad 1\leq j\leq n+2a-1;$$
$$E_{1,1}-E_{2,k}+E_{3,n}=0;$$
$$E_{2,j}-E_{2,j+1}=0,\quad 1\leq j\leq k-2;$$
$$E_{2,k-1}-E_{2,k}-E_{3,n-1}+E_{3,n}=0;$$
$$E_{3,j}-E_{3,j+1}=0,\quad 1\leq j\leq n-2;$$
$$aE_{2,k-1}-(a+1)E_{3,n-1}+E_{3,n}+aE_{4,1}=0.$$

Now, the Gale dual collection $v_{i,j}$ in $\Z^{2n+2a+k-1}$ is the following:
$$v_{0,1}=e_1+e_2,\quad v_{0,i}=e_{i+1}-e_i,\quad 2\leq i\leq n+2a-1,\quad v_{0,n+2a}=-e_{n+2a};$$
$$v_{1,1}=e_{n+2a+1}-e_1;$$
$$v_{2,1}=e_{n+2a+2}\text{ if }k\geq 3,\quad v_{2,i}=e_{n+2a+i+1}-e_{n+2a+i},\quad 2\leq i\leq k-2,$$
$$v_{2,k-1}=e_{n+2a+k}+ae_{2n+2a+k-1}-e_{n+2a+k-1},\quad v_{2,k}=-e_{n+2a+1}-e_{n+2a+k};$$
$$v_{3,1}=e_{n+2a+k+1}\text{ if }n\geq 3,\quad v_{3,i}=e_{n+2a+k+i}-e_{n+2a+k+i-1},\quad 2\leq i\leq n-2,$$
$$v_{3,n-1}=\begin{cases}-e_{n+2a+k}-e_{2n+2a+k-2}-(a+1)e_{2n+2a+k-1} & \text{if }n\geq 3;\\
-e_{n+2a+k}-(a+1)e_{2n+2a+k-1} & \text{if }n=2;\end{cases}$$
$$v_{3,n}=e_{n+2a+1}+e_{n+2a+k}+e_{2n+2a+k-1},\quad v_{4,1}=ae_{2n+2a+k-1}-e_1.$$

The vectors $v_{i,j}$ are vertices of some convex simplicial polytope $Q\subset \R^{2n+2a+k-1}$ containing zero. The maximal cones
of $\Sigma_{n,k,a}$ are cones over the facets of $Q.$ The sets of vertices of the facets are precisely the complements
to the sets of the form
$$\{v_{i,j_1}, v_{i+1,j_2}, v_{i+3,j_3}\}.$$
This describes the fan $\Sigma_{n,k,a}$ completely.

For any variety $Y$ with exceptional structure sheaf denote by $l\left(Y\right)$ the maximal length of exceptional collections of line bundles.
Clearly, if a variety $Y$ has a full exceptional collection of line bundles, then $K_0(Y)\cong \Z^{l\left(Y\right)}.$

We will prove the following result:

\begin{theo}\label{<c} For any constant $c>\frac34$ and any
$a\in\Z_{>0}$ there exist $n_0(a,c)\in\Z_{>0}$ such that for any $n\geq n_0(a,c)$ there exists $k_0(n,a,c)\in\Z_{>0}$
such that for any $k\geq k_0(n,a,c)$ we have $$l\left(Y_{n,k,a}\right)<c\rk K_0(Y_{n,k,a}).$$\end{theo}

Clearly, to obtain toric Fano's without exceptional collections of line bundles, it suffices to take $c=1.$ In this case we have the following explicit result:

\begin{theo}\label{counterex}Let $a=1,$ $n=16,$ $k\geq 386.$ Then $$l\left(Y_{n,k,a}\right)<\rk K_0(Y_{n,k,a}).$$
In particular, there does not exist a full exceptional collection of line bundles on $Y_{n,k,a}$\end{theo}

We denote by $K=K_{Y_{n,k,a}}=-\sum\limits_{i,j}E_{i,j}$ the canonical class of $Y_{n,k,a}.$ For each $i\in\Z/5\Z$ we denote by $K_{i}\subset \Pic Y_{n,k,a}$ the forbidden set corresponding to $X_{i}\cup X_{i+1}\cup X_{i+2},$ and
by $\widehat{K_i}$ the forbidden set corresponding to $X_{i+3}\cup X_{i+4},$ so that $K_{i}=K-\widehat{K_i}.$ Further,
we denote by $K_{eff}\subset \Pic Y_{n,k,a}$ the set of effective line bundles, and by $K_{neg}$ the forbidden set corresponding to $\Sigma(1),$ so that
$K_{neg}=K-K_{eff}.$ Also, we put
$$K_{all}=\bigcup\limits_{i\in\Z/5\Z}\left(K_i\cup\widehat{K_i}\right)\cup K_{eff}\cup K_{neg}.$$
Clearly, $K_{all}=K-K_{all}.$
We use identification $\Pic Y_{n,k,a}\cong\Z^3$ (by our definition of $Y_{n,k,a}$).

\begin{proof}[Proof of Theorem \ref{<c}] Suppose that $(L_1,L_2,\dots,L_m)$ is an exceptional collection of line bundles on $Y_{n,k,a}.$
Denote the coordinates of $L_i$ by $(x_i,y_i,z_i).$ Then we have
\begin{equation}\label{not_belong}(x_i-x_j,y_i-y_j,z_i-z_j)\notin K_{all},\quad\text{for }1\leq i<j\leq m.\end{equation}

Our proof consists of several steps.

{\noindent {\it Step 1.}} First we prove the following bound for the maximal difference between $z_i:$
\begin{equation}\label{amplitude}\max(z_i)-\min(z_i)\leq n+k+\left\lceil\frac{n}a\right\rceil+1.\end{equation}

Straightforward computation shows that the following holds:
\begin{equation}\label{K_eff}K_{eff}=\{(x,y,z)\mid z\geq 0, x\geq -az+\max(-y,0)\};\end{equation}
\begin{equation}\label{hat(K_1)}\widehat{K_1}=\{(x,y,z)\mid y\geq 1, z\geq 0, x\leq -n-2a-1\};\end{equation}
\begin{equation}\label{hat(K_2)}\widehat{K_2}=\{(x,y,z)\mid y\leq z-1, z\geq 0, x\leq -n-2a-1+z-y\}.\end{equation}
From \eqref{K_eff}, \eqref{hat(K_1)} and \eqref{hat(K_2)} we conclude that
\begin{equation}\label{z_positive}\{z\geq \left\lceil \frac{n}a\right\rceil+2\}\subset K_{eff}\cup \widehat{K_1}\cup\widehat{K_2}\subset K_{all}.\end{equation}
By the central symmetry, we have that
\begin{equation}\label{z_negative}\{z\leq -n-k-\left\lceil \frac{n}a\right\rceil-2\}\subset K_{neg}\cup K_1\cup K_2\subset K_{all}.\end{equation}
Combining \eqref{not_belong} with \eqref{z_positive} and \eqref{z_negative}, we obtain the desired estimate \eqref{amplitude}.
We may and will assume that $\max(z_i)=0.$ Then each of $z_i$ belongs to the interval $[-n-k-\left\lceil\frac{n}a\right\rceil-1,0].$

\smallskip

{\noindent {\it Step 2.}} Now choose some $\epsilon>0$ and consider the following functions:
$$y_{max}(z)=\max(\{y_i\mid z_i=z\}),\quad y_{min}(z)=\min(\{y_i\mid z_i=z\}),$$
which are defined for those $z$ for which there exist $i$ with $z_i=z.$
We put \begin{equation}\label{T_eps}T_{\epsilon}=\#\{z\mid y_{max}(z)-y_{min}(z)>n(1+\epsilon)\}.\end{equation}
Our goal is to prove the following upper bound on $T_{\epsilon}:$
\begin{equation}\label{T_eps_bound}T_{\epsilon}<
\frac{(\left\lceil\frac{n}a\right\rceil+2)(k+\epsilon n)}{\epsilon n(n+k+\left\lceil\frac{n}a\right\rceil+2)}.\end{equation}

According to \eqref{amplitude}, the functions $y_{max}$ and $y_{min}$ are defined for at most $(n+k+\left\lceil\frac{n}a\right\rceil+2)$
values of $z.$ Hence, by the Dirichlet's principle, there exists a residue $d\in\Z/(\left\lceil\frac{n}a\right\rceil+2)\Z$ such that
\begin{equation}\label{T_d}\#\{z\equiv d \text{ mod }(\left\lceil\frac{n}a\right\rceil+2)\mid y_{max}(z)-y_{min}(z)>n(1+\epsilon)\}\geq \frac{T_{\epsilon}(n+k+\left\lceil\frac{n}a\right\rceil+2)}{\left\lceil\frac{n}a\right\rceil+2}.\end{equation}
Fix such $d,$ and denote by $T_d$ the LHS of \eqref{T_d}.

We need another four forbidden sets written explicitly:
\begin{equation}\label{K_0}K_0=\{(x,y,z)\mid y\leq \min(-k-1,z-1),x\leq -a(y+k)-n-2a\};\end{equation}
\begin{equation}\label{hat(K_0)}\widehat{K_0}=\{(x,y,z)\mid y\geq max(1,z+n+1), x\geq -ay+2a-1\};\end{equation}
\begin{equation}\label{K_3}K_3=\{(x,y,z)\mid y\geq\max(z+n+1,1),x\leq a(y-z-n-1)-n-a-1\};\end{equation}
\begin{equation}\label{hat(K_3)}\widehat{K_3}=\{(x,y,z)\mid y\leq\min(z-1,-k-1), x\geq a(y-z)+2a\}.\end{equation}

Then it is easy to see that
\begin{equation}\label{ygeq n+1}\{(x,y,z)\mid z\leq 0,y\geq n+1\}\subset \widehat{K_0}\cup K_3,\end{equation}
\begin{equation}\label{yleq -n-k-1}\{(x,y,z)\mid z\geq -n-k, y\leq -n-k-1\}\subset K_0\cup \widehat{K_3}.\end{equation}

Now let $$\{r_1,\dots,r_{T_d}\}=\{z\equiv d \text{ mod }(\left\lceil\frac{n}a\right\rceil+2)\mid y_{max}(z)-y_{min}(z)>n(1+\epsilon)\},$$
where $r_1<r_2<\dots<r_{T_d}.$ From the definition of $r_i$ it follows that $r_i-r_j\geq \left\lceil\frac{n}a\right\rceil+2$ for $i>j.$
Hence, from \eqref{z_positive} we conclude that equalities $z_s=r_i,z_w=r_j$ imply that $\Sign(s-w)=\Sign(i-j).$ Thus,
inclusion \eqref{ygeq n+1} implies that
$$y_{max}(r_{i})-y_{min}(r_{i+1})\leq n,\quad 1\leq i\leq T_d-1.$$
By definition of $r_i,$ we obtain
$$y_{max}(r_i)-y_{max}(r_{i+1})<y_{max}(r_i)-y_{min}(r_{i+1})-n(1+\epsilon)\leq -n\epsilon,\quad 1\leq i\leq T_d-1.$$
Summing up the above inequality over $2\leq i\leq T_d-1,$ together with the inequality $y_{min}(r_2)-y_{max}(r_2)<-n(1+\epsilon),$
we obtain
\begin{equation}\label{large_difference}y_{min}(r_2)-y_{max}(y_{T_d})<-n-\epsilon n(T_d-1).\end{equation}
Further, we have $$r_2-r_d=(r_2-r_1)+(r_1-r_d)\geq \left\lceil\frac{n}a\right\rceil+2+(-n-k-\left\lceil\frac{n}a\right\rceil-1)\geq -n-k+1.$$
Combining this with \eqref{yleq -n-k-1} and \eqref{large_difference}, we get an estimate
$$T_d<\frac{k}{\epsilon n}+1.$$
This inequality, together with \eqref{T_d}, gives us the desired bound:
$$T_{\epsilon}<\frac{(\left\lceil\frac{n}a\right\rceil+2)(k+\epsilon n)}{\epsilon n(n+k+\left\lceil\frac{n}a\right\rceil+2)}.$$

\smallskip

{\noindent {\it Step 3.}} Now fix some $z'.$ We are going to prove the following upper bound on the number of $i$ with $z_i=z':$
\begin{equation}\label{z_fixed}\#\{i\mid z_i=z'\}\leq (n+2a+1)(n+3).\end{equation}

Denote by
$(L_{i_1},\dots,L_{i_t})$ the subcollection which consists of bundles with $z$-coordinate $z'$ (we assume that $t>0$). Also put
\begin{equation}(p_j,q_j):=(x_{i_j},y_{i_j}),\quad 1\leq j\leq t.\end{equation}

We need explicit descriptions of one more forbidden sets:

\begin{equation}\label{K_4}K_4=\{(x,y,z)\mid z\geq 0,x\leq -n-2a-1+\min(0,z-y)\}.\end{equation}
Write down explicitly intersections of \eqref{K_eff}, \eqref{hat(K_0)}, \eqref{hat(K_1)},\eqref{hat(K_2)} and \eqref{K_4} with the plane $\{z=0\}:$
$$M_{eff}=K_{eff}\cap \{z=0\}=\{(x,y)\mid x\geq \max(-y,0)\};$$
$$\widehat{M_0}=\widehat{K_0}\cap\{z=0\}=\{(x,y)\mid y\geq n+1, x\geq -ay+2a-1\};$$
$$\widehat{M_1}=\widehat{K_1}\cap\{z=0\}=\{(x,y)\mid y\geq 1, x\leq -n-2a-1\};$$
$$\widehat{M_2}=\widehat{K_2}\cap\{z=0\}=\{(x,y)\mid y\leq -1, x\leq -n-2a-1-y\};$$
$$M_4=K_4\cap\{z=0\}=\{(x,y)\mid x\leq -n-2a-1+\min(0,-y)\}.$$
Also put $M_{all}:=K_{all}\cap\{z=0\}.$

From \eqref{not_belong} we see that
\begin{equation}\label{not_belong2}(p_i-p_j,q_i-q_j)\notin M_{all}\quad\text{for }1\leq i<j\leq t.\end{equation}

It is easy to see that

$$\{x+y\leq -n-2a-1\}\subset \widehat{M_2}\cup M_4\subset M_{all};$$
$$\{x+y\geq n+2a+1\}\subset M_{eff}\cup \widehat{M_0}\subset M_{all}.$$
Therefore,
\begin{equation}\label{|x+y|}\{|x+y|\geq n+2a+1\}\subset M_{all}.\end{equation}

Further, note that
\begin{equation}\label{x+y=0}\{y=-x\leq 0\}\subset M_{eff}\subset M_{all},\quad \{y=-x\geq n+3\}\subset \widehat{M_0}\cup \widehat{M_1}\subset M_{all}\end{equation}
(in the second inequality, for $a\geq 2$ the set $\widehat{M_1}$ is unnecessary, and $n+3$ can be replaced by $n+1$). Combining \eqref{not_belong2} with \eqref{|x+y|} and \eqref{x+y=0}, we obtain that
$$\max(p_j+q_j)-\min(p_j+q_j)\leq n+2a,$$
and each line $x+y=d$ contains at most $n+3$ points $(p_i,q_i).$ Therefore,
$$t\leq (n+2a+1)(n+3),$$ the desired inequality \eqref{z_fixed} is proved.

Upper bounds \eqref{amplitude} and \eqref{z_fixed} are yet not sufficient for our purposes.

\smallskip

{\noindent {\it Step 4.}} With notation of Step 3, choose some $\epsilon\geq \frac{2a}n$ and make an additional assumption:
\begin{equation}\max(q_j)-\min(q_j)\leq n(1+\epsilon).\end{equation}
Under these assumptions, we will obtain another upper bound on $t:$
\begin{equation}\label{z_fixed_str}\#\{i\mid z_i=z'\}\leq (\frac34+\epsilon)n^2+(\frac32+\epsilon+a+2\epsilon a)n-a^2+a+1.\end{equation}

We may and will assume that $\max(q_i)=\max(p_i+q_i)=0.$ Then for all $1\leq i\leq t$ we have
\begin{equation}\label{parallelogramm}(p_i,q_i)\in\{-\lfloor n(1+\epsilon)\rfloor\leq q\leq 0,-n-2a\leq p+q\leq 0\}.\end{equation}
Further, choose indices $b$ and $u$ such that $p_b=\min(p_i),$ $p_u=\max(p_i).$

Suppose that $p_u-p_b> n(1+\epsilon).$ Then we have
$$(p_b-p_u,q_b-q_u)\in \widehat{M_1}\cup\widehat{M_2}\cup M_4\subset M_{all},$$
since $p_b-p_u<-n(1+\epsilon)<-n-2a$ by our assumption. Hence, $b>u.$
On the other hand, it follows from \eqref{parallelogramm} that
$$(p_u-p_b,q_u-q_b)\in M_{eff}\subset M_{all},$$
hence $u>b,$ a contradiction.

Therefore, $p_u-p_b\leq n(1+\epsilon),$ and we have
\begin{multline}(p_i,q_i)\in\{-\lfloor n(1+\epsilon)\rfloor\leq q\leq 0,-n-2a\leq p+q\leq 0,\\ p_b\leq p\leq p_b+\lfloor n(1+\epsilon)\rfloor.\}
:=Q\subset \R^2\end{multline}
We are interested in the upper bound on the number of integral points in the polygon $Q.$ Denote by $N_1$ (resp. $N_2$) the number of integral
points in the interior of $Q$ (resp. on the boundary of $Q$). By Pick's Theorem, we have
$$\Area(Q)=N_1+\frac{N_2}2-1.$$
Hence,
\begin{equation}\label{Pick}t\leq N_1+N_2=\Area(Q)+\frac{N_2}2+1.\end{equation}
First we make an estimate on $\Area(Q).$ Here we may assume that $p_b\leq 0$ (because for $p_b>0$ the polygon $Q$ is smaller than for $p_b=0$).
Then, we have
\begin{multline}\label{area(Q)}\Area(Q)=\lfloor n(1+\epsilon)\rfloor(n+2a)-\frac12(p_b^2+(n+2a+p_b)^2)\\
\leq (1+\epsilon)n(n+2a)-(\frac{n}2+a)^2=(\frac34+\epsilon)n^2+(1+2\epsilon)an-a^2.\end{multline}
Further, it is easy to make an estimate on $N_2.$ Here we also may assume that $p_b\leq 0,$ and then
\begin{multline}\label{boundary}N_2=(-p_b)+(\lfloor n(1+\epsilon)\rfloor+p_b)+(-p_b)+(n+2a+p_b)
\\+(\lfloor n(1+\epsilon)\rfloor-n-2a-p_b)+(n+2a+p_b)=2\lfloor n(1+\epsilon)\rfloor+n+2a\leq (3+2\epsilon)n+2a.\end{multline}
Combining inequality \eqref{Pick} with estimates \eqref{area(Q)} and \eqref{boundary}, we conclude that
\begin{multline*}t\leq (\frac34+\epsilon)n^2+(1+2\epsilon)an-a^2+\frac12((3+2\epsilon)n+2a)+1\\
=(\frac34+\epsilon)n^2+(\frac32+\epsilon+a+2\epsilon a)n-a^2+a+1,\end{multline*}
the desired inequality \eqref{z_fixed_str} is proved.

\smallskip

{\noindent {\it Step 5.}}

Now we apply estimates \eqref{amplitude}, \eqref{T_eps_bound}, \eqref{z_fixed} and \eqref{z_fixed_str} to finish the proof. From this moment we assume that $\frac{2a}{n}\leq \epsilon<\frac14$

First, from \eqref{amplitude}, \eqref{z_fixed}, \eqref{z_fixed_str} we obtain the upper bound on the length of our exceptional collection
\begin{multline}\label{mleq}m\leq T_{\epsilon}\cdot \left(n+2a+1\right)\left(n+3\right)\\
+\left(n+k+\left\lceil\frac{n}a\right\rceil+2-T_{\epsilon}\right)
\left(\left(\frac34+\epsilon\right)n^2+\left(\frac32+\epsilon+a+2\epsilon a\right)n-a^2+a+1\right)\\
\leq T_{\epsilon}\cdot \left(\left(\frac14-\epsilon\right)n^2+\left(a+\frac52\right)n+a^2+3a+2\right)\\+
\left(n+\left\lceil\frac{n}a\right\rceil+k+2\right)\left(\left(\frac34+\epsilon\right)n^2+\left(\frac32 a+2\right)n-a^2+a+1\right).\end{multline}

Combining \eqref{mleq} with \eqref{T_eps_bound}, we get the following inequality:
\begin{multline}\label{mleq_2}m\leq \frac{(\left\lceil\frac{n}a\right\rceil+2)(k+\epsilon n)}{\epsilon n(n+\left\lceil\frac{n}a\right\rceil+k+2)}\left(\left(\frac14-\epsilon\right)n^2+\left(a+\frac52\right)n+a^2+3a+2\right)\\
+\left(n+\left\lceil\frac{n}a\right\rceil+k+2\right)\left(\left(\frac34+\epsilon\right)n^2+\left(\frac32 a+2\right)n-a^2+a+1\right)=:E(n,k,a,\epsilon).\end{multline}

By the formula \eqref{rk_K_0},
we have that
\begin{multline}\label{K_0(Y_n,k,a)}\rk K_0(Y_{n,k,a})=(n+2a)n+k+kn(n+2a)+n+(n+2a)k\\
=n^2k+2akn+n^2+nk+2an+2ak+k+n.\end{multline}
Now, combining \eqref{K_0(Y_n,k,a)} and \eqref{mleq_2}, we can write
\begin{equation}\label{difference}c\rk K_0(Y_{n,k,a})-E(n,k,a,\epsilon)=\frac{P_2(n,a,\epsilon)k^2+P_1(n,a,\epsilon)k+P_0(n,a,\epsilon)}{\epsilon n(n+\left\lceil\frac{n}a\right\rceil+k+2)},\end{equation}
where
\begin{multline}\label{P_2}P_2(n,a,\epsilon)=\epsilon n(cn(n+2a)+cn+2ca+c-\left(\frac34+\epsilon\right)n^2\\
-\left(\frac32 a+2\right)n+a^2-a-1)
=\epsilon n((c-\frac34-\epsilon) n^2+(2ac+c-\frac32 a-2)n+2ac+c+a^2-a-1).\end{multline}

Now choose some $0<\epsilon<c-\frac34.$ By \eqref{P_2}, there exists $n_0(a,\epsilon)\geq \frac{2a}{\epsilon}$ such that for $n\geq n_0(a,\epsilon)$ we have $P_2(n,a,\epsilon)>0.$
Further, for such $n,$ according to \eqref{difference}, there exists $k_0(n,a,\epsilon)>0$ such that for $k\geq k_0(n,a,\epsilon)$
we have $c\rk K_0(Y_{n,k,a})>E(n,k,a,\epsilon).$ Finally, combining with \eqref{mleq_2}, we conclude that for such $n,k,a$
$$l(Y_{n,k,a})<c\rk K_0(Y_{n,k,a}).$$
Theorem is proved.
\end{proof}

\begin{proof}[Proof of Theorem \ref{counterex}.] We will apply the proof of the previous Theorem. Namely, by \eqref{mleq_2},
we have that $$l(Y_{16,k,1})\leq E(16,k,1,\frac18).$$
Further, a straightforward computation (solving quadratic inequality in one variable) shows that
$$\rk K_0(Y_{16,k,1})>E(16,k,1,\frac18),\quad\text{ for }k\geq 386.$$
This proves Theorem.
\end{proof}

\section{Strong exceptional collections of length at least $\frac34\rk K_0(Y).$}
\label{at_least_sect}

In this section we prove the following theorem.

\begin{theo}\label{at least}For any toric nef-Fano DM stack $Y$ with Picard number three, there exists a strong exceptional collection of line bundles on $Y$
of length at least $\frac34 \rk K_0(Y).$\end{theo}

\begin{proof}Let $\Sigma$ be a fan describing $Y.$ Then by Appendix, Theorem \ref{description}, there exists a number $t\in\Z_{>0}$ and a decomposition
$$\Sigma(1)=\bigsqcup\limits_{i\in\Z/(2t+1)}X_i$$
with $X_i\ne\emptyset,$ such that the primitive collections are precisely
$$X_i\cup X_{i+1}\cup\dots\cup X_{i+t-1},\quad i\in\Z/(2t+1)\Z.$$

Denote by $K_i$ (resp. $\widehat{K_i}$), $i\in\Z/(2t+1),$ the forbidden set corresponding to $X_i\cup X_{i+1}\cup\dots\cup X_{i+t}$
(resp. $X_{i-1}\cup X_{i-2}\cup\dots\cup X_{i-t}$), and by $K_{neg}$ the forbidden set corresponding to $\Sigma(1).$ Also put
$$K_{bad}:=\bigcup\limits_{i\in\Z/(2t+1)}(K_i\cup\widehat{K_i})\cup K_{neg}.$$
Then $K_{bad}$ is the set of all line bundles with non-zero higher cohomology (by Corollary \ref{non-trivial_cohom} and Appendix, Theorem \ref{description}).

Denote by $E_i\in \Pic Y,$ $i\in\Sigma(1),$ the invariant divisors. Put $$\widehat{\Pic}_{\R}(Y):=\Pic_{\R}(Y)/(\R\cdot K_Y).$$
Denote by $\pi:\Pic_{\R}(Y)\to \widehat{\Pic}_{\R}(Y)$ the projection, and by $\iota:\Pic(Y)\to\Pic_{\R}(Y)$ the inclusion. We will often
write $E_i$ instead of $\iota(E_i).$
Put $\widehat{E_i}:=\pi(E_i).$

Take the polytope
\begin{equation}\widehat{P}:=\sum\limits_{i\in\Sigma(1)}[0,\widehat{E_j}]\subset \widehat{\Pic}_{\R}(Y),\end{equation}
the Minkowski sum of the intervals $[0,\widehat{E_j}].$ It is easy to see that $\widehat{P}$ is centrally symmetric
with respect to zero (since $\sum\limits_{i\in\Sigma(1)}\widehat{E}_i=0$).
 Further, fix some functional $l:\Pic_{\R}(Y)\to\R$ such that $l(E_i)>0$ for $i\in\Sigma(1).$ Consider the polytope
\begin{equation}P:=\{v\in\Pic_{\R}(Y)\mid \pi(v)\in\widehat{P},\,|l(v)|\leq l(-K_Y)\}\subset\Pic_{\R}(Y).\end{equation}
It is also centrally symmetric (since $\widehat{P}$ is).
Denote by $\Int(P)$ the interior of $P.$

\begin{lemma}\label{lots_coll}For each $p\in\Pic_{\R}(Y),$ the set
$$\iota^{-1}(p+\frac12 \Int(P))\subset\Pic(Y)$$ can be ordered in such a way that it becomes a strong exceptional collection.\end{lemma}

\begin{proof} It suffices to prove that for any $L_1,L_2\in \iota^{-1}(p+\frac12 \Int(P))$ we have $H^{>0}(L_2L_1^{-1})=0.$
Further, this would follow from the absence of intersection:
\begin{equation}\label{intersection}\Int(P)\cap\iota(K_{bad})=\emptyset.\end{equation}
Now, choose some functionals $l_i:\widehat{\Pic}_{\R}(Y)\to\R,$ $i\in\Z/(2t+1),$ such that
$$l_i(\widehat{E_j})\begin{cases}\geq 0 &\text{for }j\in X_i\cup X_{i+1}\cup\dots\cup X_{i+t};\\
\leq 0 & \text{for }j\in X_{i-1}\cup X_{i-2}\cup\dots\cup X_{i-t}\end{cases}$$
(they exist by Appendix, Theorem \ref{description}).
Then $$\pi^*l_i(\iota(D))\leq l_i(-\sum\limits_{j\in X_i\cup X_{i+1}\cup\dots\cup X_{i+t}}\widehat{E_j})\quad\text{for }D\in K_i,$$
and $$\pi^*l_i(v)>l_i(-\sum\limits_{j\in X_i\cup X_{i+1}\cup\dots\cup X_{i+t}}\widehat{E_j})\quad\text{for }v\in \Int(P).$$
Therefore, $\Int(P)\cap\iota(K_i)=\emptyset.$ Analogously, $\Int(P)\cap\iota(\widehat{K_i})=\emptyset.$ Finally, we have
$$l(\iota(D))\leq l(K_Y)\quad\text{for }D\in K_{neg},$$
and
$$l(v)>l(K_Y)\quad\text{for }v\in\Int(P).$$
Therefore, $\Int(P)\cap\iota(K_{neg})=\emptyset.$ This proves \eqref{intersection}. Lemma is proved.
\end{proof}

Fix the volume form $\omega$ on $\Pic_{\R}(Y)$ such that $|\omega(D_1,D_2,D_3)|=|\Ker(\iota)|$ for
any $\Z$-basis $(D_1,D_2,D_3)$ of $\iota(\Pic(Y)).$ Below we take volumes with respect to this form.

\begin{lemma}\label{Fubini_lemma}There exists a point $p\in\Pic_{\R}$ such that
\begin{equation}\label{Fubini}|\iota^{-1}(p+\frac12 \Int(P))|\geq \frac18\Vol(P).\end{equation}\end{lemma}

\begin{proof}This follows easily from Fubini's theorem. Namely, fix some $\Z$-basis $(D_1,D_2,D_3)$ of $\iota(\Pic(Y)).$
Put $C:=[0,1)^3,$ and take the measure $dx\wedge dy\wedge dz$ on $C.$ Further, take a measure on $\Z^3$
such that the measure of each point equals to $|\Ker(\iota)|.$ Then we have an isomorphism of spaces with measure:
$$C\times\Z^3\stackrel{\sim}{\to}\Pic_{\R}(Y),\,\, ((t_1,t_2,t_3),(m_1,m_2,m_3))\mapsto (t_1+m_1)D_1+(t_2+m_2)D_2+(t_3+m_3)D_3.$$
Denote by $q:\Pic_{\R}(Y)\to C$ the resulting projection. Then, by Fubini's theorem,
$$\frac18\Vol(P)=\int\limits_{C}|(|\Ker(\iota)|\cdot |q^{-1}(t_1,t_2,t_3)\cap \frac12 \Int(P)|).$$
Since $\Vol(C)=1,$ there exists $(t_1,t_2,t_3)\in C$ such that
$$\frac18\Vol(P)\leq |\Ker(\iota)|\cdot |q^{-1}(t_1,t_2,t_3)\cap \frac12 \Int(P)|=|\iota^{-1}(t_1D_1+t_2D_2+t_3D_3+\frac12 \Int(P))|.$$
Hence, \eqref{Fubini} holds for $p=t_1D_1+t_2D_2+t_3D_3.$ Lemma is proved.
\end{proof}

Now we will obtain the lower bound on $\Vol(P).$

\begin{lemma}\label{Vol/rk_K_0}The following inequality holds:
$$\Vol(P)\geq 6\rk K_0(Y).$$\end{lemma}

\begin{proof}Take the volume form $\widehat{\omega}$ on $\widehat{\Pic}_{\R}(Y)$ such that
$$\omega=\pi^*(\widehat{\omega})\wedge dl.$$
Using the form $\widehat{\omega}$ we identify $\Lambda^2(\widehat{\Pic}_{\R}(Y))\cong\R.$
Then for any $G_1,G_2,G_3\in \Pic_{\R}(Y)$ we have
\begin{equation}\omega(G_1,G_2,G_3)=l(G_1)\pi(G_2)\wedge\pi(G_3)+l(G_2)\pi(G_3)\wedge\pi(G_1)+l(G_3)\pi(G_1)\wedge\pi(G_2).\end{equation}

Now put
\begin{equation}W_i:=\sum\limits_{j\in X_i} E_j,\quad i\in\Z/(2t+1).\end{equation}
Put $\widehat{W}_i:=\pi(W_i).$ We may and will assume  that
$$\widehat{W}_i\wedge \widehat{W}_{i+j}\geq 0\quad\text{for }i\in\Z/(2t+1),\, 1\leq j\leq t$$
(otherwise we multiply $\omega$ and $\widehat{\omega}$ by $(-1)$). Then we have
\begin{equation}\label{geq0_1}(\widehat{W}_i+\dots+\widehat{W}_{i+t})\wedge (\widehat{W}_{i+1}+\dots+\widehat{W}_{i+t})=\sum\limits_{j=1}^t \widehat{W}_i\wedge \widehat{W}_{i+j}\geq 0.\end{equation}
Analogously,
\begin{equation}\label{geq0_2}(\widehat{W}_i+\dots+\widehat{W}_{i+t-1})\wedge (\widehat{W}_i+\dots+\widehat{W}_{i+t})\geq 0.\end{equation}
It follows from \eqref{geq0_1} and \eqref{geq0_2} that
\begin{multline}\Vol(\widehat{P})\geq \sum\limits_{i\in\Z/(2t+1)}\frac12((\widehat{W}_i+\dots+\widehat{W}_{i+t})\wedge (\widehat{W}_{i+1}+\dots+\widehat{W}_{i+t})\\+
(\widehat{W}_i+\dots+\widehat{W}_{i+t-1})\wedge (\widehat{W}_i+\dots+\widehat{W}_{i+t}))=\sum_{\substack{i\in\Z/(2t+1),\\
1\leq j\leq t}}\widehat{W}_i\wedge\widehat{W}_{i+j}.\end{multline}
Hence,
\begin{equation}\label{Vol(P)}\Vol(P)=2l(-K_Y)\Vol(\widehat{P})\geq 2(\sum\limits_{i\in\Z/(2t+1)}l(W_i))(\sum_{\substack{i\in\Z/(2t+1),\\
1\leq j\leq t}}\widehat{W}_i\wedge\widehat{W}_{i+j}).\end{equation}

We are going to obtain a similar upper bound on $\rk K_0(Y).$ From \eqref{K_0_general}, Lemma \ref{volumes} and Lemma \ref{vol_D_tors} it follows that $\rk K_0(Y)$ equals to the sum of
$|\omega(E_{u_1},E_{u_2},E_{u_3})|$ over all subsets $\{u_1,u_2,u_3\}\subset\Sigma(1)$ which are complements to the maximal cones.
These are precisely sets $\{u_1,u_2,u_3\}\subset\Sigma(1).$ such that for some $i\in\Z/(2t+1),$ $1\leq j_1,j_2\leq t,$ $j_1+j_2>t$ we have
$$u_1\in X_i,\quad u_2\in X_{i+j_1},\quad u_3\in X_{i-j_2}.$$
Therefore, we have
\begin{equation}\label{rk_K_0_stack}\rk K_0(Y)=
\sum\limits
_{i\in\Z/(2t+1)}l(W_i)(\sum_{\substack{1\leq j_1,j_2\leq t,\\
j_1+j_2>t}}\widehat{W}_{i+j_1}\wedge\widehat{W}_{i-j_2}). \end{equation}

{\noindent {\bf Sublemma.}} {\it Suppose that we are given with collection of vectors $g_i\in\widehat{\Pic}_{\R}(Y),$
 $i\in\Z/(2t+1),$ with $\sum\limits_{i\in\Z/(2t+1)}g_i=0.$ Suppose that that there exist non-zero functionals $f_i:\widehat{\Pic}_{\R}(Y)\to\R$ such that
$$f_i(g_j)\begin{cases}\geq 0 & \text{for }j=i,i+1,\dots,i+t\\
\leq 0 & \text{for }j=i-1,i-2,\dots,i-t.\end{cases}$$ Assume that
$$g_i\wedge g_{i+j}\geq 0\quad\text{for }i\in\Z/(2t+1),\,1\leq j\leq t.$$ Then
\begin{equation}\sum_{\substack{r\in\Z/(2t+1),\\
1\leq j\leq t}}g_r\wedge g_{r+j}\geq 3\sum_{\substack{1\leq j_1,j_2\leq t,\\
j_1+j_2>t}}g_{j_1}\wedge g_{-j_2}.\end{equation}}

\begin{proof}First note that
$$\sum_{\substack{1\leq j_1,j_2\leq t,\\
j_1+j_2>t}}g_{j_1}\wedge g_{-j_2}=\sum\limits_{j=1}^t(g_{-j}\wedge (\sum\limits_{k=1}^tg_{k-j}+\sum\limits_{k=j+1}^tg_{-k}))=\sum_{\substack{1\leq j_1\leq t,\\
0\leq j_2\leq t-j_1}}g_{-j_1}\wedge g_{j_2}.$$
Similarly,
$$\sum_{\substack{1\leq j_1,j_2\leq t,\\
j_1+j_2>t}}g_{j_1}\wedge g_{-j_2}=\sum_{\substack{0\leq j_1\leq t,\\
1\leq j_2\leq t-j_1}}g_{-j_1}\wedge g_{j_2}.$$
Therefore,
\begin{multline*}\sum_{\substack{r\in\Z/(2t+1),\\
1\leq j\leq t}}g_r\wedge g_{r+j}-3\sum_{\substack{1\leq j_1,j_2\leq t,\\
j_1+j_2>t}}g_{j_1}\wedge g_{-j_2}=\sum_{\substack{r\in\Z/(2t+1),\\
1\leq j\leq t}}g_r\wedge g_{r+j}-\sum_{\substack{1\leq j_1,j_2\leq t,\\
j_1+j_2>t}}g_{j_1}\wedge g_{-j_2}\\-\sum_{\substack{1\leq j_1\leq t,\\
0\leq j_2\leq t-j_1}}g_{-j_1}\wedge g_{j_2}-\sum_{\substack{0\leq j_1\leq t,\\
1\leq j_2\leq t-j_1}}g_{-j_1}\wedge g_{j_2}\\=\sum_{\substack{1\leq j_1<j_2\leq t}}g_{-j_2}\wedge g_{-j_1}
+\sum_{\substack{1\leq j_1<j_2\leq t}}g_{j_1}\wedge g_{j_2}-\sum_{\substack{1\leq j_1\leq t,\\
1\leq j_2\leq t-j_1}}g_{-j_1}\wedge g_{j_2}.\end{multline*}

Thus, we are left to prove the following inequality:
\begin{equation}\label{main}\sum_{\substack{1\leq j_1<j_2\leq t}}g_{-j_2}\wedge g_{-j_1}
+\sum_{\substack{1\leq j_1<j_2\leq t}}g_{j_1}\wedge g_{j_2}-\sum_{\substack{1\leq j_1\leq t,\\
1\leq j_2\leq t-j_1}}g_{-j_1}\wedge g_{j_2}\geq 0.\end{equation}

We proceed by induction on $t.$ For $t=1,$ the LHS of \eqref{main} equals to zero, and there is nothing to prove.

Suppose that \eqref{main} is proved in the case $t\leq m.$ Let us prove it for $t=m+1.$
Consider the cases

{\it Case 1:} for some $j\ne 0$ we have $g_j=0.$ We may and will assume that $j\in \{1,\dots,t\}$ (by the symmetry).
Form another collection $g_i'\in\widehat{\Pic}_{\R}(Y),$
$i\in\Z/(2t-1),$ given by the formula
$$(g_0',g_1',\dots,g_{2t-2}')=(g_0,g_1,\dots,g_{j-1},g_{j+1},\dots,g_{j+t}+g_{j+t+1},\dots,g_{2t})$$
(in the case $j=t$ we have $g_0'=g_{2t}+g_{2t+1}$). Clearly, this new collection satisfies the assumptions of Sublemma, and one computes that

\begin{multline}\label{reduce}(\sum_{\substack{1\leq j_1<j_2\leq t}}g_{-j_2}\wedge g_{-j_1}
+\sum_{\substack{1\leq j_1<j_2\leq t}}g_{j_1}\wedge g_{j_2}-\sum_{\substack{1\leq j_1\leq t,\\
1\leq j_2\leq t-j_1}}g_{-j_1}\wedge g_{j_2})-\\
-(\sum_{\substack{1\leq j_1<j_2\leq t-1}}g_{-j_2}'\wedge g_{-j_1}'
+\sum_{\substack{1\leq j_1<j_2\leq t-1}}g_{j_1}'\wedge g_{j_2}'-\sum_{\substack{1\leq j_1\leq t-1,\\
1\leq j_2\leq t-1-j_1}}g_{-j_1}'\wedge g_{j_2}')=\\
=g_{j+t}\wedge g_{j+t+1}\geq 0.\end{multline}
Then, inequality \eqref{main} follows from \eqref{reduce} and inductive hypothesis.

{\it Case 2:} we have $g_j\ne 0$ for $j\ne 0,$ but for some $j\ne 0,t+1$ we have $g_j=-\kappa g_{t+j},$ $\kappa>0.$
We may and will assume that $\kappa\geq 1.$ We form another collection $g_i'\in\widehat{\Pic}_{\R}(Y),$
$i\in\Z/(2t+1),$ given by the formula
$$g_i':=\begin{cases}g_j+g_{t+j} &\text{for }i=j;\\
0 &\text{for }i=j+t;\\
g_i &\text{for }i\in(\Z/(2t+1))\setminus\{j,j+t\}.\end{cases}$$
This new collection obviously satisfies the assumptions of Sublemma and
we have
\begin{multline}\label{reduce2}(\sum_{\substack{1\leq j_1<j_2\leq t}}g_{-j_2}\wedge g_{-j_1}
+\sum_{\substack{1\leq j_1<j_2\leq t}}g_{j_1}\wedge g_{j_2}-\sum_{\substack{1\leq j_1\leq t,\\
1\leq j_2\leq t-j_1}}g_{-j_1}\wedge g_{j_2})-\\
-(\sum_{\substack{1\leq j_1<j_2\leq t}}g_{-j_2}'\wedge g_{-j_1}'                                     +\sum_{\substack{1\leq j_1<j_2\leq t}}g_{j_1}'\wedge g_{j_2}'-\sum_{\substack{1\leq j_1\leq t,\\
1\leq j_2\leq t-j_1}}g_{-j_1}'\wedge g_{j_2}')=\\
=\begin{cases}(g_{j+1}+\dots+g_{j+t-1})\wedge g_{j+t}\geq 0 &\text{if }j\in\{1,\dots,t\};\\
g_{j+t}\wedge (g_{j-1}+\dots+g_{j-t})\geq 0 &\text{if }j\in\{t+2,\dots,2t\}.\end{cases}.\end{multline}
Hence, we are reduced to the Case 1.

We are left with

{\it Case 3:} we have that $g_j,$ $j\ne 0,$ are pairwise linearly independent. Take (unique) $\kappa>0$ such that
$g_{2t}+\kappa g_t$ is linearly dependent with $g_{t-1}.$ Put $$\kappa':=\min(\kappa,1).$$
Form another collection $g_i'\in\widehat{\Pic}_{\R}(Y),$
$i\in\Z/(2t+1),$ given by the formula
$$g_i':=\begin{cases}g_{2t}+\kappa'g_t &\text{for }i=2t;\\
(1-\kappa')g_t &\text{for }i=t;\\
g_i &\text{for }i\in(\Z/(2t+1))\setminus\{t,2t\}.\end{cases}$$
This new collection obviously satisfies the assumptions of Sublemma and
we have
\begin{multline}\label{reduce3}(\sum_{\substack{1\leq j_1<j_2\leq t}}g_{-j_2}\wedge g_{-j_1}
+\sum_{\substack{1\leq j_1<j_2\leq t}}g_{j_1}\wedge g_{j_2}-\sum_{\substack{1\leq j_1\leq t,\\
1\leq j_2\leq t-j_1}}g_{-j_1}\wedge g_{j_2})-\\
-(\sum_{\substack{1\leq j_1<j_2\leq t}}g_{-j_2}'\wedge g_{-j_1}'
+\sum_{\substack{1\leq j_1<j_2\leq t}}g_{j_1}'\wedge g_{j_2}'-\sum_{\substack{1\leq j_1\leq t,\\
1\leq j_2\leq t-j_1}}g_{-j_1}'\wedge g_{j_2}')=\\
=\kappa'g_{t}\wedge(g_{t+1}+\dots+g_{2t-1})\geq 0.\end{multline}
We are reduced either to the Case 1 (if $\kappa'=1$) or to the Case 2 (if $\kappa'<1$).

In all cases the inductive statement is proved for $t=m+1.$ Sublemma is proved.
\end{proof}

Finally, from \eqref{Vol(P)}, Sublemma and \eqref{rk_K_0_stack} we get the following chain of equalities and inequalities:
\begin{multline*}\Vol(P)\geq 2(\sum\limits_{i\in\Z/(2t+1)}l(W_i))(\sum_{\substack{i\in\Z/(2t+1),\\
1\leq j\leq t}}\widehat{W}_i\wedge\widehat{W}_{i+j})\\
\geq 6\sum\limits_{i\in\Z/(2t+1)}l(W_i)(\sum_{\substack{1\leq j_1,j_2\leq t,\\
j_1+j_2>t}}\widehat{W}_{i+j_1}\wedge\widehat{W}_{i-j_2})=6\rk K_0(Y).\end{multline*}
Lemma is proved.\end{proof}

From Lemmas \ref{lots_coll}, \ref{Fubini_lemma} and \ref{Vol/rk_K_0} it follows that for some $p\in\Pic(\R)$
the set $\iota^{-1}(p+\frac12\Int(P))$ can be ordered in such a way that it becomes a strong exceptional collection of line bundles
of length at least $\frac18\cdot 6\rk K_0(Y)=\frac34\rk K_0(Y).$ Theorem is proved.\end{proof}

\appendix

\section{Smooth projective toric DM stacks with Picard number three}
\label{Appendix}

Here we describe combinatorial structure of the fans defining smooth projective toric DM stacks with Picard number three.
Let $Y$ be such a stack, $\Sigma$ a fan in a lattice $N$ defining it, and $v_i\in N,$ $i\in\Sigma(1)$ are marked vectors
on one-dimensional cones. Denote by $E_i\in\Pic(Y),$ $i\in\Sigma(1),$ the invariant divisors.

\begin{theo}\label{description}1) There exists $t\geq 1$ and a decomposition $$\Sigma(1)=\bigsqcup\limits_{i\in\Z/(2t+1)}X_i,$$
such that the primitive collections
are precisely
\begin{equation}\label{primitive}X_i\cup X_{i+1}\cup\dots\cup X_{i+t-1},\quad i\in\Z/(2t+1).\end{equation}

2) If $Y$ is Fano (resp. nef-Fano) then there exist non-zero functionals $l_i:\Pic_{\R}(Y)\to\R,$ $i\in\Z/(2t+1)$ such that $l_i(K_Y)=0$ and
$l_i(E_j)>0$ (resp. $l_i(E_j)\geq 0$) for $j\in X_i\cup X_{i+1}\cup\dots\cup X_{i+t},$ and
$l_i(E_j)<0$ (resp. $l_i(E_j)\leq 0$) for $j\in X_{i-1}\cup X_{i-2}\cup\dots\cup X_{i-t}.$

3) The subsets $I\subset\Sigma(1)$ such that $|C_I|$ has non-trivial reduced homology are precisely the following:
$\emptyset,$ $\Sigma(1),$ $X_i\cup X_{i+1}\cup\dots\cup X_{i+t},$ $X_{i-1}\cup X_{i-2}\cup\dots\cup X_{i-t},$ $i\in\Z/(2t+1).$ \end{theo}

\begin{proof} 1) Take any $\Q$-ample line bundle $L\in\Pic(Y).$ Then by \cite{FLTZ} (Theorem 4.4), it can be written
as
$$L=\sum\limits_{i\in\Sigma(1)}a_iE_i\in\Pic_{\Q}(Y),\quad a_i\in\Q_{>0},$$
and the polytope $$\bigcup\limits_{\langle v_{i_1},\dots,v_{i_{\dim Y}}\rangle\in\Sigma(\dim Y)}
\conv(\frac{v_{i_1}}{a_1},\dots,
\frac{v_{i_{\dim Y}}}{a_{\dim Y}},0)$$
is convex, with vertices being precisely all $v_i.$

Denote by $\pi:\Pic_{\R}(Y)\to \Pic_{\R}(Y)/(\R\cdot L)$ the projection.

For each set $\{u_1,u_2,u_3\}\subset\Sigma(1)$
which is a completion to some maximal cone, we have by Proposition \ref{compl_to_faces}
\begin{equation}\label{ample}0\in\R_{>0}\pi(E_{u_1})+\R_{>0}\pi(E_{u_2})+\R_{>0}\pi(E_{u_3}).\end{equation}
Moreover, for any $u_1,u_2\in \Sigma(1)$ we have again by Proposition \ref{compl_to_faces}
\begin{equation}\label{not_belong_stack}0\notin\R_{>0}\pi(E_{u_1})+\R_{>0}\pi(E_{u_2}).\end{equation}

Now the desired subsets $X_i\subset\Sigma(1)$ are defined as maximal subsets $X\subset \Sigma(1)$ with the following property:
$$-E_k\notin\sum\limits_{j\in X}\R_{\geq 0} E_j=:A_X\quad\text{for each }k\in\Sigma(1).$$
It follows from \eqref{not_belong_stack} and \eqref{ample} that:

(i) the number of such $X$ is odd and at least three, say $2t+1;$

(ii) they are disjunctive and their union is $\Sigma(1);$

(iii) for different $X_i$ and $X_j$ we have that
$A_{X_i}\cap A_{X_j}=\{0\}.$

We order the $X_i$ cyclically in such a way that the cones $A_{X_0},\dots A_{X_{2t}}$
go in the anti-clockwise direction (for some orientation on $\Pic_{\R}(Y)/(\R\cdot L)$). It is clear from \eqref{ample} that primitive collections are precisely as
in \eqref{primitive}.

2) If $Y$ is Fano, then the statement follows from the proof of 1).
Let $Y$ be nef-Fano. Then  there exist sequence $a_{\sigma,n}>0,$ with $\sigma\in\Sigma(1),$ $n\geq 1,$ such that
$$\lim\limits_{n\to\infty}a_{\sigma,n}=1,\quad\text{and}\quad L_n=\sum\limits_{\sigma\in\Sigma(1)}a_{\sigma,n}E_{\sigma}\text{ is }\Q-\text{ample for all }n.$$
Denote by $l_{i,n},$ $i\in\Z/(2t+1),$ the desired functionals for $L_n$ instead of $-K_Y.$ Assume that $||l_{i,n}||=1$ for some norm on $\Pic_{\R}(Y).$
Then each sequence $\{l_{i,n}\}_{n=1}^{\infty}$ has some partial limit $l_i.$ The functionals $l_i$ satisfy the desired properties.

3) It suffices to remind that the following holds:

(iv) if $\bar{H}_{\cdot}(C_I)\ne 0$ and $I\ne\emptyset,$ then $I$ is a union of primitive collections (Lemma \ref{union});

(v) if $\bar{H}_{\cdot}(C_I)\ne 0,$ then also $\bar{H}_{\cdot}(C_{\Sigma(1)\setminus I})\ne 0.$

(vi) if $I$ is a primitive collection, then $\bar{H}_{\cdot}(C_I)\ne 0.$

The assertion immediately follows from (iv), (v), (vi) and part 1).
\end{proof}

\end{document}